\documentclass[english,12pt,a4paper,twoside]{smfart}
\usepackage{amsfonts,smfthm}
\usepackage{amsmath}
\usepackage{stmaryrd}
\usepackage{amssymb}
\usepackage{graphics}
\usepackage{babel}
\usepackage{enumerate}
\usepackage{eufrak}
\usepackage{euscript}

\usepackage{newlfont}
\usepackage{latexsym}
\usepackage{graphicx}
\usepackage{a4wide}

\usepackage{cite}        
\usepackage{url}         

\title{On a combinatorial problem of \\ Erd\H{o}s, Kleitman and Lemke}
\author{Benjamin Girard}
\thanks{B. GIRARD, 
IMJ, \'{E}quipe Combinatoire et Optimisation, Universit\'{e} Pierre et Marie Curie~(Paris $6$), $4$ place Jussieu, 75005 Paris, France
(\textit{e-mail}: bgirard@math.jussieu.fr)}

\address{
IMJ, \'{E}quipe Combinatoire et Optimisation, Universit\'{e} Pierre et Marie Curie~(Paris $6$), $4$ place Jussieu, 75005 Paris, France.}
\email{bgirard@math.jussieu.fr}

\theoremstyle{plain}
\newtheorem{theorem}{Theorem}[section]
\newtheorem{proposition}[theorem]{Proposition}
\newtheorem{lem}[theorem]{Lemma}
\newtheorem{corollary}[theorem]{Corollary}
\newtheorem{conjecture}{Conjecture}

\theoremstyle{definition}

\theoremstyle{remark}

\def\cnp[#1,#2]{\begin{pmatrix} #1 \\#2 \end{pmatrix}}
\begin{document}

\maketitle \setcounter{page}{1} \vspace{0.0cm}

\begin{abstract} 
In this paper, we study a combinatorial problem originating in the following conjecture of Erd\H{o}s and Lemke: given any sequence of $n$ divisors of $n$, repetitions being allowed, there exists a subsequence the elements of which are summing to $n$. 
This conjecture was proved by Kleitman and Lemke, who then extended the original question to a problem on a zero-sum invariant in the framework of finite Abelian groups.
Building among others on earlier works by Alon and Dubiner and by the author, our main theorem gives a new upper bound for this invariant in the general case, and provides its right order of magnitude.
\end{abstract}

\vspace{-0.7cm}
\section{Introduction}
\label{Section Introduction}
Let $G$ be a finite Abelian group, written additively. 
If $G$ is cyclic of order $n$, it will be denoted by $C_n$. 
In the general case, we can decompose $G$ as a direct product of cyclic groups $C_{n_1} \oplus \cdots \oplus C_{n_r}$ such that $1 < n_1 \mid \cdots \mid n_r \in \mathbb{N}$, where $r$ and $n_r$ are respectively called the rank and exponent of $G$.
Usually, the exponent of $G$ is simply denoted by $\exp(G)$.
The order of an element $g$ of $G$ will be written $\text{ord}(g)$ and for every divisor $d$ of $\exp(G)$,
we denote by $G_d$ the subgroup of $G$ consisting of all elements of order dividing $d$: 
$$G_d=\left\{x \in G \text{ } | \text{ } dx=0\right\}.$$ 

\medskip
In this paper, any finite sequence $S=(g_1,\dots,g_\ell)$ of $\ell$ elements of $G$  will be called a \em sequence \em over $G$ of \em length \em $|S|=\ell$. We will also denote by $\sigma(S)$ the sum of all elements contained in $S$, which will be referred to as a \em zero-sum sequence \em whenever $\sigma(S)=0$. 

\medskip
Given a sequence $S$ over $G$, we denote by $S_d$ the subsequence of $S$ consisting of all elements of order $d$ contained in $S$, and by $\mathsf{k}(S)$ the 
\em cross number \em of $S=(g_1,\dots,g_\ell)$, which is defined as follows:
$$\mathsf{k}(S)=\displaystyle\sum_{i=1}^{\ell} \frac{1}{\text{ord}(g_i)}.$$
 
\medskip
By $\mathsf{t}(G)$ we denote the smallest integer $t \in \mathbb{N}^*$ such that every sequence $S$ over $G$ of length $|S| \geq t$ contains a non-empty zero-sum subsequence $S' \mid S$ with $\mathsf{k}(S') \leq 1$. Such a subsequence will be called a \em tiny zero-sum subsequence\em.

\medskip
The investigations on $\mathsf{t}(G)$ originate in the following conjecture, addressed by Erd\H{o}s and Lemke in the late eighties (see \cite[Introduction]{KleitLem89}). Is it true that out of $n$ divisors of $n$, repetitions being allowed, one can always find a certain number of them that sum up to $n$? Motivated by this conjecture, Kleitman and Lemke \cite[Theorem 1]{KleitLem89} proved the following stronger result.

\begin{theorem}[Kleitman and Lemke \cite{KleitLem89}] 
\label{Kleitman et Lemke}
For any given integers $a_1,\dots,a_n$ there is a non-empty subset $I \subseteq \llbracket 1,n \rrbracket$ such that
$$n \mid \displaystyle\sum_{i \in I} a_i \hspace{0.2cm} \text{ and } \hspace{0.2cm} \displaystyle\sum_{i \in I} \gcd(a_i,n) \leq n.$$
\end{theorem}

Meanwhile, Lagarias and Saks framed a graph-theoretic approach, called graph pebbling, thanks to which the problem of Erd\H{o}s and Lemke could be reduced to the study of a combinatorial game, played with pebbles on the vertices of a simple graph. In this context, Chung \cite[Theorem 6]{Chung89} found a new elegant proof of Theorem \ref{Kleitman et Lemke}, and, under one extra assumption made on the prime factors of $n$, Denley \cite[Theorem 1]{Denley} could sharpen this theorem. Graph pebbling led to a rapidly growing literature as well as many open problems. The interested reader is referred to the surveys \cite{Hurlbert99,Hurlbert05} which contain many references on the subject.

\medskip
In addition, let us underline that, with our notation, Theorem \ref{Kleitman et Lemke} simply asserts that $\mathsf{t}(C_n) \leq n,$ and this upper bound is easily seen to be optimal. 
In the general case, Kleitman and Lemke \cite[Section 3]{KleitLem89} conjectured that $\mathsf{t}(G) \leq |G|$ holds for every finite Abelian group $G$. 
This conjecture was proved, using tools from zero-sum theory, by Geroldinger \cite{Gero93}. An alternative proof was then found by Elledge and Hurlbert \cite[Theorem 2]{Elledge05} using graph pebbling. 

\medskip
In this paper, we prove that in the general case of finite Abelian groups, the currently known upper bound on $\mathsf{t}(G)$ can be improved significantly. 
More precisely, our main theorem shows that, for finite Abelian groups of fixed rank, $\mathsf{t}(G)$ grows linearly in the exponent of $G$, which gives the correct order of magnitude. 
To do so, we prove that $\mathsf{t}(G)$ can be bounded above using a result of Alon and Dubiner \cite{Alodub95} on the following classical invariant in zero-sum theory.

\medskip
By $\eta(G)$ we denote the smallest integer $t \in \mathbb{N}^*$ such that every sequence $S$ over $G$ of length $|S| \geq t$ contains a non-empty zero-sum subsequence $S' \mid S$ with $|S'| \leq \exp(G)$. Such a subsequence is called a \em short zero-sum subsequence\em.

\medskip
Since $\mathsf{k}(T) \leq 1$ implies $|T| \leq \exp(G)$, one can notice that $\eta(G) \leq \mathsf{t}(G)$ always holds. 
Let us also mention that $\eta(G)$ is one out of many invariants studied because of their arithmetical applications. 
In this respect, the interested reader is referred to \cite{GeroKoch05,GaoGero06,GeroRuzsa09} for comprehensive surveys on non-unique factorization theory. 
For instance, it is known \cite[Theorem 5.8.3]{GeroKoch05} that for all integers $m,n \in \mathbb{N}^*$ such that $m \mid n$, one has
\begin{equation}
\label{eta rank two} 
\eta(C_m \oplus C_n)=2m+n-2.
\end{equation}
In particular, the equality $\mathsf{t}(C_n)=\eta(C_n)=n$ holds. 
Moreover, and even if little is known on the exact value of $\eta(G)$ for finite Abelian groups of rank $r \geq 3$, its behavior is better understood than it is for $\mathsf{t}(G)$, and the following key result, due to Alon and Dubiner \cite[Theorem 1.1]{Alodub95}, will be extensively used throughout this article.

\begin{theorem}[Alon and Dubiner \cite{Alodub95}]  
\label{Alon c_r} 
For every $r \in \mathbb{N}^*$ there exists a constant $c_r > 0$ such that for every $n \in \mathbb{N}^*$, one has 
$$\eta\left(C^r_n\right) \leq c_r(n-1)+1.$$
\end{theorem}

From now on, we will identify $c_r$ with its smallest possible value in Theorem \ref{Alon c_r}. 
On the one hand, a natural construction shows that $c_r \geq (2^r-1)$. 
Indeed, if $(e_1,\dots,e_r)$ is a basis of $C^r_n$ with $\text{ord}(e_i)=n$ for all $i \in \llbracket 1,r \rrbracket$, it is easily checked that the sequence $S$ consisting of $n-1$ copies of $\sum_{i \in I} e_i$ for each non-empty subset $I \subseteq \llbracket 1,n \rrbracket$ contains no short zero-sum subsequence, so that 
\begin{equation}
\label{lower bound}
(2^r-1)(n-1)+1 \leq \eta(C^r_n) \leq \mathsf{t}(C^r_n).
\end{equation}
In particular, one has 
$\mathsf{t}(C^r_2)=\eta(C^r_2)=2^r$. 
For the time being, the exact value of $\mathsf{t}(G)$ is known for cyclic groups and elementary $2$-groups only.

\medskip
On the other hand, the method used in \cite{Alodub95} yields $c_r \leq \left( cr\ln r\right)^r$, where $c > 0$ is an  absolute constant. 
Also, it readily follows from (\ref{eta rank two}) that it is possible to choose $c_1=1$ and $c_2=3$, with equality in Theorem \ref{Alon c_r}, and it is conjectured in \cite{Alodub95} that there actually is an absolute constant $d > 0$ such that $c_r \leq d^r$ for all $r \geq 1$. For a complete account on $\eta(G)$, see \cite{EdelGero06,GaoSchmid06} and the references contained therein.

\section{New results and plan of the paper}
\label{Section New results and plan of the paper}
Let $\mathcal{P}=\left\{p_1=2<p_2=3<\cdots \right\}$ be the set of prime numbers. 
Given a positive integer $n$, let $\mathcal{D}_n$ be the set of its positive divisors. 
By $P^-(n)$ and $P^+(n)$, we denote the smallest and greatest prime elements of $\mathcal{D}_n$ respectively, with the convention $P^-(1)=P^+(1)=1$. 
Finally, the $p$-adic valuation of $n$ will be denoted by $\nu_p(n)$.

\medskip
In this paper, we prove that in the general case of finite Abelian groups, the currently known upper bound on $\mathsf{t}(G)$ can be improved significantly. 

\medskip
Our starting point will be to prove it first in the case of finite Abelian $p$-groups.
For this purpose, a classical variant of $\mathsf{t}(G)$, introduced by Geroldinger and Schneider in \cite{GeroSchnCross94}, will be studied in Section \ref{Section rho(G)}. 
Even though a single by-product of this investigation (see Corollary \ref{t(G) in the p-group case}) will effectively be used in subsequent sections, we include this study in full, since it may be of interest in view of arithmetical applications. 

\medskip
Then, in Section \ref{Section main theorem}, we prove the main theorem of this paper. 
This theorem provides a new upper bound for $\mathsf{t}(G)$, which depends on the rank and exponent of $G$ only. 
It is proved thanks to an appropriate partition of the divisor lattice of $\exp(G)$. In addition, an interesting special case is the one of finite Abelian groups of rank two, where this theorem can be applied specifying $c_2=3$. 

\begin{theorem}
\label{main result}
Let $G$ be a finite Abelian group of rank $r$ and exponent $n$. Then
$$\mathsf{t}(G) \leq c_r\displaystyle\sum_{d \mid n} \left(\frac{d}{P^{+}(d)^{\nu_{P^+(d)}(d)}}-1\right) + c_r\left(n-1\right) + 1.$$
\end{theorem}

\medskip
In Section \ref{Section modified sum of divisors}, we then obtain, as a corollary of Theorem \ref{main result}, the following Alon and Dubiner type upper bound for $\mathsf{t}(G)$. 
In particular, it is easily deduced from (\ref{lower bound}) that this upper bound has the right order of magnitude.
\begin{theorem} 
\label{t(G) est lineaire en n}
For every $r \in \mathbb{N}^*$ there exists a constant $d_r > 0$ such that, for every finite Abelian group $G$ of rank $r$ and exponent $n$, one has
$$\mathsf{t}(G) \leq d_r(n-1)+1.$$
\end{theorem}

The qualitative upper bound of Theorem \ref{t(G) est lineaire en n} is voluntarily given in a form allowing to stress the connection between $d_r$ and $c_r$. 
In this regard, a key argument in our proof of Theorem \ref{t(G) est lineaire en n} actually comes from a simple, albeit somewhat surprising, property of the following arithmetic function
$$f(n)=\displaystyle\sum_{d \mid n} \frac{d}{P^{+}(d)}.$$
Indeed, one can show that $f(n) \leq n$ always holds (see Proposition \ref{modified sum of divisors}). 
Consequently, and even if no particular effort has been made to optimize the constant in Theorem \ref{t(G) est lineaire en n}, it turns out that $d_r$ can be chosen to satisfy $1 \leq d_r/c_r \leq 2$, thus being at most twice as large as the best possible constant we could hope for. 

\medskip
Finally, in Section \ref{Section concluding remarks}, we propose and discuss two open problems on $\mathsf{t}(G)$.

\section{On a variant of $\mathsf{t}(G)$}
\label{Section rho(G)}
Let $G$ be a finite Abelian group, and $d',d \in \mathbb{N}^*$ be two integers such that $d' \mid d \mid \exp(G)$.
This section is devoted to the following variant of $\mathsf{t}(G)$, which was first introduced by Geroldinger and Schneider \cite[Section 3]{GeroSchnCross94} (see also \cite[Section 5.7]{GeroKoch05}):
$$\rho(G)=\max\{\mathsf{k}(S) \mid S \text{ contains no tiny zero-sum subsequence} \}.$$

Using a compression argument, we will obtain a new upper bound for $\rho(G)$, which has the correct order of magnitude and applies to any finite Abelian group $G$. 
For this purpose, we consider the following invariant, which was introduced in \cite{Girard09}.

\medskip
By $\eta_{(d',d)}(G)$ we denote the smallest integer $t \in \mathbb{N}^*$ such that every sequence $S$ over $G_d$ of length $|S| \geq t$ contains a non-empty subsequence $S' \mid S$ of length 
$|S'| \leq d'$ and with sum in $G_{d/d'}$.

\medskip
The numbers $\eta_{(d',d)}(G)$ and $\eta(G)$ are closely related to each other. First, note that the two definitions coincide when $d'=d=\exp(G)$.
In addition, and as shown in \cite[Proposition 3.1]{Girard09}, there exists a subgroup $G_{\upsilon(d',d)} \subseteq G$ such that $\eta_{(d',d)}(G)=\eta(G_{\upsilon(d',d)})$. 
In order to define this particular subgroup $G_{\upsilon(d',d)}$ properly, we introduce the following notation. 
Given the decomposition of $G$ as a product of cyclic groups 
$$G \simeq C_{n_1} \oplus \cdots \oplus C_{n_r}, \text{ with } 1 < n_1 \mid \cdots \mid n_r \in \mathbb{N},$$
we set, for all $i \in \llbracket 1,r \rrbracket$, 
$$A_i=\gcd(d',n_i), \text{ }
B_i=\frac{\mathrm{lcm}(d,n_i)}{\mathrm{lcm}(d',n_i)}, \text{ } \upsilon_i(d',d)=\frac{A_i}{\gcd(A_i,B_i)}.$$
Therefore, whenever $d$ divides $n_i$, we have $\upsilon_i(d',d)=\gcd(d',n_i)=d'$, and in particular $\upsilon_r(d',d)=d'$. We can now state our result on $\mathsf{\eta}_{(d',d)}(G)$.

\begin{proposition} 
\label{propmarrantegenerale} 
Let $G \simeq C_{n_1} \oplus \cdots \oplus C_{n_r}$, with $1 < n_1 \mid \cdots \mid n_r \in \mathbb{N}$, be a finite Abelian group and $d',d \in \mathbb{N}^*$ be such that $d' \mid d \mid \exp(G)$. Then
$$\mathsf{\eta}_{(d',d)}(G)=\mathsf{\eta}\left(C_{\upsilon_1(d',d)} \oplus \cdots \oplus C_{\upsilon_r(d',d)}\right).$$
\end{proposition}

Using Theorem \ref{Alon c_r}, our Proposition \ref{propmarrantegenerale} yields the following useful estimate on the numbers $\eta_{(d',d)}(G)$.

\begin{proposition}
\label{estimate of generalized eta numbers}
For every $r \in \mathbb{N}^*$ there exists a constant $c_r > 0$ such that for every finite Abelian group $G$ of rank $r$, and every $d' \mid d \mid \exp(G)$, one has 
$$\eta_{(d',d)}(G) \leq c_r(d'-1)+1.$$
\end{proposition}

\begin{proof}
Let us consider the group 
$$G_{\upsilon(d',d)}=C_{\upsilon_1(d',d)} \oplus \cdots \oplus C_{\upsilon_r(d',d)}.$$
On the one hand, Proposition \ref{propmarrantegenerale} states that $\eta_{(d',d)}(G)=\eta(G_{\upsilon(d',d)})$. 
On the other hand, it is easily seen that $\upsilon_i(d',d) \mid d'$ for all $i \in \llbracket 1,r \rrbracket$ and $\upsilon_r(d',d)=d'$, which implies that $G_{\upsilon(d',d)}$ is a subgroup of $C^r_{d'}$ of exponent $d'$.
Now, since $\eta(H) \leq \eta(G)$ holds for all groups $H \subseteq G$ such that $\exp(H)=\exp(G)$, Theorem \ref{Alon c_r} gives
$$\eta_{(d',d)}(G) = \eta(G_{\upsilon(d',d)}) \leq \eta(C^r_{d'}) \leq c_r(d'-1)+1.$$ 
\end{proof}

Now, we can prove the main theorem of this section.

\begin{theorem}
\label{Profils}
Let $G$ be a finite Abelian group of exponent $n$ and $\mathcal{D}_{n}=\left\{d_1,\dots,d_{m}\right\}$. 
Let also $S$ be a sequence over $G$ containing no tiny zero-sum subsequence, reaching the maximum $\mathsf{k}(S)=\rho(G)$, and being of minimal length regarding this property. 
Then, the $m$-tuple $x=\left(x_{d_1},\dots,x_{d_m}\right)$, where $x_d=|S_d|$ for all $d$, is an element of the polytope
$$\mathbb{P}_G=\{x \in \mathbb{N}^m \text{ } | \text{ } f_d(x) \geq 0, \text{ } d \in \mathcal{D}_n\},$$ 
where 
$$f_d(x) = \displaystyle\min_{d' \in \mathcal{D}_d\backslash\{1\}} \left(\mathsf{\eta}_{(d',d)}(G)-1 - x_d\right).$$
\end{theorem}

\begin{proof}
Let $S$ be a sequence over $G$ containing no tiny zero-sum subsequence, satisfying $\mathsf{k}(S)=\rho(G)$ and being of minimal length regarding this property. 
Suppose also that the $m$-tuple $x=(x_{d_1},\dots,x_{d_m})$, where $x_d=|S_d|$ for all $d$, is not an element of the polytope $\mathbb{P}_G$.

\medskip
Then, there exists $d_0 \in \mathcal{D}_{n}$ such that $f_{d_0}(x)<0$, which means there is a $d'_0 \in \mathcal{D}_{d_0}\backslash\{1\}$ satisfying $x_{d_0} \geq \eta_{(d'_0,d_0)}(G)$. 
So, the sequence $S$ contains $X$ elements $g_1,\dots,g_X$ of order $d_0$, with $1 < X \leq d'_0$, the sum $\sigma$ of which is an element of order $\tilde{d_0}$ dividing $d_0/d'_0$.

\medskip
Let $T$ be the sequence obtained from $S$ by replacing these $X$ elements by their sum.  
Reciprocally, for every subsequence $T'$ of $T$ containing $\sigma$, let us denote by $\varphi(T')$ the subsequence of $S$ obtained from $T'$ by replacing $\sigma$ by $g_1,\dots,g_X$. 
In particular, note that 
\begin{eqnarray*}
\mathsf{k}(\varphi(T')) & = & \mathsf{k}(T') - \frac{1}{\text{ord}(\sigma)} + \displaystyle\sum^X_{i=1} \frac{1}{\text{ord}(g_i)} \\
                         & = & \mathsf{k}(T') + \frac{X}{d_0} - \frac{1}{\tilde{d_0}} \\
                         & \leq & \mathsf{k}(T') + \frac{X}{d_0} - \frac{d'_0}{d_0} \\
                         & \leq & \mathsf{k}(T').
\end{eqnarray*}

First, $T$ contains no tiny zero-sum subsequence. 
Indeed, if $T'$ is a non-empty zero-sum subsequence of $T$, then either $T'$ contains $\sigma$ so that $\varphi(T')$ is a non-empty zero-sum subsequence of $S$, which yields 
$$\mathsf{k}(T') \geq \mathsf{k}(\varphi(T')) > 1,$$
or $T'$ does not contain $\sigma$, which implies that $T'$ is a subsequence of $S$ and $\mathsf{k}(T') > 1$. 

\medskip
Second, it follows from the equality $\varphi(T)=S$ that $\mathsf{k}(T) \geq \mathsf{k}(S)=\rho(G)$. 
Therefore, $T$ is a sequence over $G$ containing no tiny zero-sum subsequence such that $\mathsf{k}(T)=\rho(G)$ and $|T|=|S|-X+1<|S|$, a contradiction.
\end{proof}

Keeping in mind the notation used in Theorem \ref{Profils}, we now obtain the following immediate corollary, giving a general upper bound for $\rho(G)$, expressed as the solution of an integer linear program.

\begin{corollary} 
\label{Upper bound for rho(G)} 
For every finite Abelian group $G$ of exponent $n$, one has
$$\rho(G) \leq \displaystyle\max_{x \in \mathbb{P}_G} \displaystyle\sum_{d \mid n} \frac{x_d}{d}.$$
\end{corollary}

It is now possible, using Proposition \ref{estimate of generalized eta numbers}, to deduce the following quantitative result from Theorem \ref{Profils}.

\begin{theorem} 
\label{majoration rho(G)}
Let $G$ be a finite Abelian group of rank $r$ and exponent $n$. Then 
$$\rho(G) \leq c_r \displaystyle\sum_{d \mid n} \frac{P^-(d)-1}{d}.$$
\end{theorem}

\begin{proof}
Using Theorem \ref{Profils} and Proposition \ref{estimate of generalized eta numbers}, we indeed obtain
\begin{eqnarray*}
\rho(G) & \leq & \displaystyle\sum_{d \mid n} \frac{\eta_{(P^-(d),d)}(G)-1}{d}\\
              & \leq & c_r \displaystyle\sum_{d \mid n} \frac{P^-(d)-1}{d}.
\end{eqnarray*}
\end{proof}

Theorem \ref{majoration rho(G)} actually improves on the best known upper bound $\rho(G) \leq |G|/P^-(n)$ proved by Geroldinger and Schneider \cite[Lemma 2.1]{GeroSchnCross94}.
In addition, a simple study of the arithmetic function involved in our result (see \cite[Lemma 5.1]{Girard09}) shows there exists a constant $\delta_r >0$ such that, for every finite Abelian group $G$ of rank $r$ and exponent $n$, one has $\rho(G) \leq \delta_r\omega(n)$, where $\omega(n)$ denotes the number of distinct prime divisors of $n$.

\medskip
On the other hand, Theorem \ref{majoration rho(G)} provides us with the following useful result on $\mathsf{t}(G)$ in the case of finite Abelian $p$-groups. 

\begin{corollary}
\label{t(G) in the p-group case}
Let $p \in \mathcal{P}$. Then, for all $\alpha,r \in \mathbb{N}^*$, one has  
$$\mathsf{t}(C^r_{p^{\alpha}}) \leq c_r\left(p^{\alpha}-1\right)+1.$$
\end{corollary}

\begin{proof}
Let $G \simeq C^r_{p^{\alpha}}$, where $\alpha,r \in \mathbb{N}^*$ and $p \in \mathcal{P}$. Using Theorem \ref{majoration rho(G)}, one has
$$\rho(G) \leq c_r \displaystyle\sum_{i=1}^\alpha \left(\frac{p-1}{p^{i}}\right)=c_r \left(\frac{p^\alpha-1}{p^\alpha}\right).$$
Then, for every sequence $S$ over $G$ containing no tiny zero-sum subsequence, one has
$$\frac{|S|}{p^\alpha} \leq \mathsf{k}(S) \leq \rho(G) \leq c_r \left(\frac{p^\alpha-1}{p^\alpha}\right),$$
which gives $|S| \leq c_r \left(p^\alpha-1\right)$ and completes the proof.

\end{proof}
As already mentioned, an interesting special case is the one of finite Abelian groups of rank two. 
In this case, specifying $c_2=3$ in Corollary \ref{t(G) in the p-group case} implies that, for all primes $p$ and $\alpha \in \mathbb{N}^*$, the equality 
\begin{equation}
\label{t rank two}
\mathsf{t}\left(C_{p^\alpha} \oplus C_{p^\alpha}\right)=3p^{\alpha}-2=\eta\left(C_{p^\alpha} \oplus C_{p^\alpha}\right)
\end{equation} 
holds. Building on the case where $G$ has prime power exponent, we can now turn to the general case of finite Abelian groups.

\section{Proof of the main theorem}
\label{Section main theorem}
Let $G$ be a finite Abelian group of exponent $n=q^{\alpha_1}_1 \cdots q^{\alpha_\ell} _\ell$, with $q_1 < \cdots < q_\ell$, and let $S$ be a sequence over $G$. 
We consider the following partition:
$$\mathcal{D}_n \backslash \{1\}=\bigcup^{\ell}_{i=1} \mathcal{A}_i, \text{ where } \mathcal{A}_i=\{d \in \mathcal{D}_n \backslash \{1\} : P^+(d)=q_i\}.$$
In particular, for all $d \in \mathcal{A}_i$, one has $d \leq \Delta_i=q^{\alpha_1}_1 \cdots q^{\alpha_i} _i$. 
Now, for every $d \in \mathcal{A}_i$, we denote by $k_d$ the smallest integer such that
$$|S_{d}| < k_d\frac{d}{q^{\nu_{q_i}(d)}_i} + \eta_{\left(\frac{d}{q^{\nu_{q_i}(d)}_i},d\right)}(G).$$

The number $k_d$ has the following combinatorial interpretation.
\begin{lem}
\label{lego} 
Let $G$ be a finite Abelian group of exponent $n$, and $d \in \mathcal{A}_i$. 
Let also $S$ be a sequence over $G$. 
Then $S_d$ contains at least $k_d$ disjoint non-empty subsequences $S'_1,\dots,S'_{k_d}$ such that, for all $j \in \llbracket 1,k_d \rrbracket$,
$$\sigma(S'_j) \in G_{q^{\nu_{q_i}(d)}_i} \text{ and } \mathsf{k}(S'_j) \leq \frac{1}{q^{\nu_{q_i}(d)}_i}.$$
\end{lem}

\begin{proof}
First, we shall prove by induction on $k \in \llbracket 0,k_d \rrbracket$ that $S_d$ contains at least $k$ disjoint non-empty subsequences $S'_1,\dots,S'_k$ such that, for all $j \in \llbracket 1,k \rrbracket$, 
$$\sigma(S'_j) \in G_{q^{\nu_{q_i}(d)}_i} \text{ and } |S'_j| \leq \frac{d}{q^{\nu_{q_i}(d)}_i}.$$
If $k=0$ then this statement is clearly true. 
Now, assume that the statement holds for some $k \in \llbracket 0,k_d-1 \rrbracket$, and let us prove that it holds for $k+1$ as well.
By the induction hypothesis, we already know that $S_d$ contains $k$ disjoint non-empty subsequences $S'_1,\dots,S'_k$ such that, for all $j \in \llbracket 1,k \rrbracket$, 
$$\sigma(S'_j) \in G_{q^{\nu_{q_i}(d)}_i} \text{ and } |S'_j| \leq \frac{d}{q^{\nu_{q_i}(d)}_i}.$$
Moreover, the sequence $T_d$ obtained from $S_d$ by deleting all elements of $S'_1,\dots,S'_k$ satisfies
\begin{eqnarray*}
|T_d| & = & |S_d| - \displaystyle\sum^k_{j=1} |S'_j| \\
          & \geq & \left(k_d-1\right)\frac{d}{q^{\nu_{q_i}(d)}_i} + \eta_{\left(\frac{d}{q^{\nu_{q_i}(d)}_i},d\right)}(G) -k\frac{d}{q^{\nu_{q_i}(d)}_i}\\ 
          &  \geq & \eta_{\left(\frac{d}{q^{\nu_{q_i}(d)}_i},d\right)}(G), 
\end{eqnarray*}
so that $S_d$ contains a non-empty subsequence $S'_{k+1}$ disjoint from $S'_1,\dots,S'_k$ such that
$$\sigma(S'_{k+1}) \in G_{q^{\nu_{q_i}(d)}_i} \text{ and } |S'_{k+1}| \leq \frac{d}{q^{\nu_{q_i}(d)}_i},$$
which completes the induction. Therefore, $S_d$ contains $k_d$ disjoint non-empty subsequences $S'_1,\dots,S'_{k_d}$ such that, for all $j \in \llbracket 1,k_d \rrbracket$, 
$$\sigma(S'_j) \in G_{q^{\nu_{q_i}(d)}_i} \text{ and } |S'_j| \leq \frac{d}{q^{\nu_{q_i}(d)}_i}.$$
In addition, for all $j \in \llbracket 1,k_d \rrbracket$, one clearly has
$$\mathsf{k}(S'_j) = \frac{|S'_j|}{d} \leq \frac{1}{q^{\nu_{q_i}(d)}_i},$$
and the lemma is proved.
\end{proof}

\begin{corollary}
\label{lego corollary} 
Let $G$ be a finite Abelian group of exponent $n$, and $d \in \mathcal{A}_i$. 
Let also $S$ be a sequence over $G$. 
Then $S_d$ contains at least $k_d$ disjoint non-empty subsequences $S'_1,\dots,S'_{k_d}$ such that, for all $j \in \llbracket 1,k_d \rrbracket$,
$$\sigma(S'_j) \in G_{q^{\nu_{q_i}(n)}_i} \text{ and } \mathsf{k}(S'_j) \leq \frac{1}{\text{\em ord\em}\left(\sigma\left(S'_j\right)\right)}.$$
\end{corollary}

\begin{proof} 
By Lemma \ref{lego}, $S_d$ contains at least $k_d$ disjoint non-empty subsequences $S'_1,\dots,S'_{k_d}$ such that, for all $j \in \llbracket 1,k_d \rrbracket$,
$$\sigma(S'_j) \in G_{q^{\nu_{q_i}(d)}_i} \text{ and } \mathsf{k}(S'_j) \leq \frac{1}{q^{\nu_{q_i}(d)}_i}.$$
Thus, the desired result directly follows from the fact that, for all $j \in \llbracket 1,k_d\rrbracket$, one has
$$\sigma(S'_j) \in G_{q^{\nu_{q_i}(d)}_i} \subseteq G_{q^{\nu_{q_i}(n)}_i}.$$
\end{proof}

\begin{lem}
\label{legobis} 
Let $G$ be a finite Abelian group of exponent $n$, and let $S$ be a sequence over $G$ containing no tiny zero-sum subsequence. 
Then, for all $i \in \llbracket 1,\ell \rrbracket$, 
$$\displaystyle\sum_{d \in \mathcal{A}_i } k_d \leq \mathsf{t}\left(G_{q^{\nu_{q_i}(n)}_i}\right)-1.$$
\end{lem}

\begin{proof} 
Let us set $t=\mathsf{t}\left(G_{q^{\nu_{q_i}(n)}_i}\right)$, and assume that
$$\displaystyle\sum_{d \in \mathcal{A}_i } k_d \geq t.$$
Then, it follows from Corollary \ref{lego corollary} that $S$ contains $t$
disjoint non-empty subsequences $S'_1,\dots,S'_t$ such that, for all $j \in \llbracket 1,t \rrbracket$,
$$\sigma(S'_j) \in G_{q^{\nu_{q_i}(n)}_i} \text{ and } \mathsf{k}(S'_j) \leq \frac{1}{\text{ord}\left(\sigma\left(S'_j\right)\right)}.$$
Now, let us consider the sequence $T$ over $G_{q^{\nu_{q_i}(n)}_i}$ defined by
$$T=\prod^t_{j=1} \sigma(S'_j).$$
Since $T$ is a sequence of length $|T| = t$, then it contains a tiny zero-sum subsequence $T' \mid T$.
In other words, there exists a non-empty subset $J \subseteq \llbracket 1,t \rrbracket$ such that
$$T'=\prod_{j \in J} \sigma(S'_j).$$

\medskip
Now, let us set
$$S'=\prod_{j \in J}S'_j.$$ 
Since $\sigma(S')=\sigma(T')=0$, then $S'$ is a non-empty zero-sum subsequence of $S$. 
In addition, we have the following chain of inequalities:
\begin{eqnarray*}
\mathsf{k}(S') &    =   & \displaystyle\sum_{j \in J} \mathsf{k}(S'_j)\\
                         &   \leq   & \displaystyle\sum_{j \in J} \frac{1}{\text{ord}\left(\sigma\left(S'_j\right)\right)}\\ 
                         &   =    & \mathsf{k}(T')\\ 
                         & \leq & 1.
\end{eqnarray*} 
Therefore, $S$ contains a tiny zero-sum subsequence, and the proof is complete.
\end{proof}

We can now prove the main theorem of this paper.

\begin{proof}
[Proof of Theorem \ref{main result}] 
Let $S$ be a  sequence over $G$ containing no tiny zero-sum subsequence. For every $i \in \llbracket 1,\ell \rrbracket$, Corollary \ref{t(G) in the p-group case} and Lemma \ref{legobis} yield
$$\displaystyle\sum_{d \in \mathcal{A}_i } k_d \leq \mathsf{t}\left(G_{q^{\nu_{q_i}(n)}_i}\right)-1 \leq c_r\left(q^{\nu_{q_i}(n)}_i-1\right),$$
so that, setting $\Delta_0=1$, one obtains
\begin{eqnarray*}
\displaystyle\sum_{d \in \mathcal{A}_i } \left(|S_d| -\left( \eta_{\left(\frac{d}{q^{\nu_{q_i}(d)}_i},d\right)}(G)-1\right)\right) &   \leq   &  \displaystyle\sum_{d \in \mathcal{A}_i }  k_d\frac{d}{q^{\nu_{q_i}(d)}_i}\\
                                                                                                                                                         &  \leq   & \Delta_{i-1} \displaystyle\sum_{d \in \mathcal{A}_i }  k_d\\
                                                                                                                                                         &  \leq   & c_r\Delta_{i-1}\left(q^{\nu_{q_i}(n)}_i-1\right)\\
                                                                                                                                                         &    =     & c_r\left(\Delta_i-\Delta_{i-1}\right).
\end{eqnarray*}
Now, using Proposition \ref{estimate of generalized eta numbers}, we have
\begin{eqnarray*}
|S| & = & \displaystyle\sum^\ell_{i=1} \displaystyle\sum_{d \in \mathcal{A}_i} |S_d|\\
      & = & \displaystyle\sum^\ell_{i=1}\displaystyle\sum_{d \in \mathcal{A}_i } \left( \eta_{\left(\frac{d}{q^{\nu_{q_i}(d)}_i},d\right)}(G) -1 \right) + \displaystyle\sum^\ell_{i=1}\displaystyle\sum_{d \in \mathcal{A}_i } \left(|S_d| - 
      \left(\eta_{\left(\frac{d}{q^{\nu_{q_i}(d)}_i},d\right)}(G) - 1 \right) \right)\\
          & \leq & c_r \displaystyle\sum^\ell_{i=1}\displaystyle\sum_{d \in \mathcal{A}_i } \left(\frac{d}{q^{\nu_{q_i}(d)}_i} - 1\right) + c_r \displaystyle\sum^\ell_{i=1}\left(\Delta_i-\Delta_{i-1}\right)\\
      &   =   & c_r\displaystyle\sum_{d \mid n} \left(\frac{d}{P^{+}(d)^{\nu_{P^+(d)}(d)}}-1\right) + c_r\left(n-1\right),
\end{eqnarray*}
which completes the proof of the theorem.
\end{proof}

\section{A special sum of divisors}
\label{Section modified sum of divisors}
In this section, we derive Theorem \ref{t(G) est lineaire en n} from the study of a particular arithmetic function. 
We recall $\mathcal{P}=\{p_1=2 < p_2=3 < \cdots\}$ denotes the set of prime numbers, and start by proving the following easy lemma.
\begin{lem} \label{primes} For every integer $\ell \geq 1$, one has
$$\displaystyle\prod^{\ell}_{i=1} \left(1 + \frac{1}{p_i-1} \right) \leq p_{\ell+1}-1.$$
\end{lem}

\begin{proof} To start with, let us prove the following statement by induction on $\ell \geq 1$.
$$\displaystyle\prod^{\ell}_{i=2} \left(1 + \frac{1}{2(i-1)} \right) \leq \ell.$$
One can readily notice that this statement is true for $\ell=1$ and $\ell=2$. 
Assume now that the statement holds for some $\ell \geq 2$. 
Then, let us show that it holds for $\ell+1$ also.
Indeed, 
\begin{eqnarray*}
\displaystyle\prod^{\ell+1}_{i=2} \left(1 + \frac{1}{2(i-1)} \right) & = & \displaystyle\prod^{\ell}_{i=2} \left(1 + \frac{1}{2(i-1)} \right) \left(1 + \frac{1}{2\ell} \right)\\
                                                                                                            & \leq  & \ell\left(1 + \frac{1}{2\ell} \right)\\
                                                                                                            & \leq & \ell+1,
\end{eqnarray*}
and we are done. The desired result now follows from the following chain of inequalities, using the trivial bound $p_\ell \geq 2\ell-1$, for all $\ell \geq 2$, in the following fashion.
\begin{eqnarray*}
\displaystyle\prod^{\ell}_{i=1} \left(1 + \frac{1}{p_i-1} \right)  & = & 2\displaystyle\prod^{\ell}_{i=2} \left(1 + \frac{1}{p_i-1} \right)\\
                                                                                                     & \leq & 2\displaystyle\prod^{\ell}_{i=2} \left(1 + \frac{1}{2(i-1)} \right)\\
                                                                                                     & \leq  & 2\ell\\
                                                                                                     & = & (2(\ell+1)-1)-1\\
                                                                                                     & \leq & p_{\ell+1}-1.
\end{eqnarray*}
\end{proof}

We can now prove the main result of this section. 
\begin{proposition} 
\label{modified sum of divisors} 
Let $f(n)=\displaystyle\sum_{d \mid n} \frac{d}{P^{+}(d)}$. For every integer $n \geq 1$, one has $f(n) \leq n$.
\end{proposition}

\begin{proof}
Writing $n=q^{\alpha_1}_1 \cdots q^{\alpha_\ell}_\ell$, where $q_1 < \cdots < q_\ell$, we prove the desired result by induction on $\ell = \omega(n) \geq 1$.

\medskip
If $\ell=1$, then one has
\begin{eqnarray*}
f(q^{\alpha_1}_1) & = & \displaystyle\sum_{d \mid q^{\alpha_1} _1}\frac{d}{P^{+}(d)}\\ 
                                & = & 1 + \displaystyle\sum^{\alpha_1}_{i =1} q^{i-1}_1\\
                                & = & 1 + \left(\frac{q^{\alpha_1}_1-1}{q_1-1}\right),
\end{eqnarray*}
so that we now have
\begin{eqnarray*}
\frac{f(q^{\alpha_1}_1)}{q^{\alpha_1}_1} & =  & \frac{1}{q^{\alpha_1}_1} + \frac{1}{q_1-1}\left(\frac{q^{\alpha_1}_1-1}{q^{\alpha_1}_1}\right)\\
                                                                         & \leq  & \frac{1}{q^{\alpha_1}_1} + \left(\frac{q^{\alpha_1}_1-1}{q^{\alpha_1}_1}\right)\\
                                                                         &    =   & 1,
\end{eqnarray*}
and we are done.

\medskip
Assume now that the statement holds true for some $\ell \geq 1$.
Then, setting $\sigma(n)=\sum_{d \mid n} d$, we obtain the following equalities.
\begin{eqnarray*}
f(q^{\alpha_1}_1 \cdots q^{\alpha_{\ell+1}}_{\ell+1}) & = & \displaystyle\sum_{d \mid q^{\alpha_1} _1 \cdots q^{\alpha_{\ell+1}}_{\ell+1}}  \frac{d}{P^{+}(d)}\\ 
                                & = & \displaystyle\sum_{d \mid q^{\alpha_1}_1 \cdots q^{\alpha_\ell} _\ell} \frac{d}{P^{+}(d)} 
                                + \displaystyle\sum^{\alpha_{\ell+1}}_{i=1} \displaystyle\sum_{d \mid q^{\alpha_1}_1 \cdots q^{\alpha_\ell} _\ell}  \frac{d q^{i}_{\ell+1}}{q_{\ell+1}}\\ 
                                & = & f(q^{\alpha_1}_1 \cdots q^{\alpha_\ell} _\ell) + \sigma(q^{\alpha_1}_1 \cdots q^{\alpha_\ell} _\ell) \displaystyle\sum^{\alpha_{\ell+1}}_{i=1}  q^{i-1}_{\ell+1}\\
                                & = & f(q^{\alpha_1}_1 \cdots q^{\alpha_\ell} _\ell) + \sigma(q^{\alpha_1}_1 \cdots q^{\alpha_\ell} _\ell) \left(\frac{q^{\alpha_{\ell+1}}_{\ell+1}-1}{q_{\ell+1}-1}\right)\\                                                              
\end{eqnarray*}
so that 
\begin{eqnarray*}
\frac{f(q^{\alpha_1}_1 \cdots q^{\alpha_{\ell+1}}_{\ell+1})}{q^{\alpha_1}_1 \cdots q^{\alpha_{\ell+1}}_{\ell+1}}  & = & \frac{f(q^{\alpha_1}_1 \cdots q^{\alpha_\ell} _\ell)}{q^{\alpha_1}_1 \cdots q^{\alpha_\ell}_\ell} \frac{1}{q^{\alpha_{\ell+1}}_{\ell+1}} + \frac{\sigma(q^{\alpha_1}_1 \cdots q^{\alpha_\ell} _\ell)}{q^{\alpha_1}_1 \cdots q^{\alpha_\ell}_\ell} \frac{1}{\left(q_{\ell+1}-1\right)} \frac{q^{\alpha_{\ell+1}}_{\ell+1}-1}{q^{\alpha_{\ell+1}}_{\ell+1}}.                                                                       
\end{eqnarray*}
First, by the induction hypothesis, we have
$$\frac{f(q^{\alpha_1}_1 \cdots q^{\alpha_\ell} _\ell)}{q^{\alpha_1}_1 \cdots q^{\alpha_\ell}_\ell} \leq 1.$$
Second, since $\sigma(n)$ is multiplicative, Lemma \ref{primes} yields
\begin{eqnarray*}
\frac{\sigma(q^{\alpha_1}_1 \cdots q^{\alpha_\ell} _\ell)}{q^{\alpha_1}_1 \cdots q^{\alpha_\ell}_\ell} & = & \frac{\sigma(q^{\alpha_1}_1)}{q^{\alpha_1}_1}  \cdots \frac{\sigma(q^{\alpha_\ell} _\ell)}{q^{\alpha_\ell}_\ell}\\ 
                                                                                                                                                            & \leq & \displaystyle\prod^{\ell}_{i=1} \left(1 + \frac{1}{q_i-1} \right)\\ 
                                                                                                                                                            & \leq & \displaystyle\prod^{\ell}_{i=1} \left(1 + \frac{1}{p_i-1} \right)\\ 
                                                                                                                                                            & \leq & p_{\ell+1}-1\\
                                                                                                                                                            & \leq & q_{\ell+1}-1. 
\end{eqnarray*}
Thus, we obtain
$$\frac{f(q^{\alpha_1}_1 \cdots q^{\alpha_{\ell+1}}_{\ell+1})}{q^{\alpha_1}_1 \cdots q^{\alpha_{\ell+1}}_{\ell+1}} \leq \frac{1}{q^{\alpha_{\ell+1}}_{\ell+1}} + \frac{q^{\alpha_{\ell+1}}_{\ell+1}-1}{q^{\alpha_{\ell+1}}_{\ell+1}} = 1,$$
which completes the proof.                                                                                                                                                                                        
\end{proof}

As an immediate corollary of Proposition \ref{modified sum of divisors}, we now prove Theorem \ref{t(G) est lineaire en n}.

\begin{proof}
[Proof of Theorem \ref{t(G) est lineaire en n}]
Let $G$ be a finite Abelian group of rank $r$ and exponent $n$. Then, by Theorem \ref{main result} and Proposition \ref{modified sum of divisors}, one has 
\begin{eqnarray*}
\mathsf{t}(G) & \leq & c_r\displaystyle\sum_{d \mid n} \left(\frac{d}{P^{+}(d)^{\nu_{P^+(d)}(d)}}-1\right) + c_r\left(n-1\right) + 1\\ 
                       & \leq & c_r\displaystyle\sum_{d \mid n} \left(\frac{d}{P^{+}(d)}-1\right) + c_r\left(n-1\right) + 1\\  
                       &   \leq  & 2c_r\left(n - 1\right) +1.
\end{eqnarray*}
\end{proof}

\section{A few concluding remarks}
\label{Section concluding remarks}
As previously stated, the exact value of $\mathsf{t}(G)$ is currently known for cyclic groups and elementary $2$-groups only. 
In this context, the special case of finite Abelian groups of rank two is of particular interest, and the following conjecture appears to be inviting. 

\begin{conjecture}
\label{conjecture on the rank two case}
For all integers $m,n \in \mathbb{N}^*$ such that $m \mid n$, one has
$$\mathsf{t}(C_m \oplus C_n) = 2m + n - 2.$$
\end{conjecture}

If true, the statement of Conjecture \ref{conjecture on the rank two case} would nicely extend the theorem of Kleitman and Lemke. 
In view of equality (\ref{t rank two}), this conjecture readily holds true for all groups $G \simeq C_{p^\alpha} \oplus C_{p^\alpha}$,
where $p \in \mathcal{P}$ and $\alpha \in \mathbb{N}^*$. 
Moreover, note that $\mathsf{t}(C_m \oplus C_n) \geq 2m + n - 2$ easily follows from (\ref{eta rank two}). 

\medskip 
Even though far less is known on the exact value of $\eta(G)$ for finite Abelian groups of higher rank, it would be worth knowing how close it actually is to $\mathsf{t}(G)$ in the general case.
A first step in this direction is the following lemma, which gives a simple property shared by all finite Abelian groups for which $\mathsf{t}(G)=\eta(G)$ holds.

\begin{lem}
\label{eta of a subgroup}
Let $G$ be a finite Abelian group such that $\mathsf{t}(G)=\eta(G)$. 
Then for every subgroup $H$ of $G$, one has $\eta(H) \leq \eta(G)$.
\end{lem}

\begin{proof}
Let $G$ be as in the statement of the lemma, and let $H$ be a subgroup of $G$. 
By definition, the inequality $\mathsf{t}(H) \leq \mathsf{t}(G)$ holds. 
Therefore,
$$\eta(H) \leq \mathsf{t}(H) \leq \mathsf{t}(G) = \eta(G),$$
and the required result is proved.
\end{proof}

However, it turns out that the above property is restrictive enough to guarantee that for every $r \geq 4$, there exists a finite Abelian group $G$ of rank $r$ for which $\mathsf{t}(G) > \eta(G)$.
The proof of this fact actually relies on the following key invariant in zero-sum combinatorics.
Given a finite Abelian group $G$, let $\mathsf{D}(G)$ denote the smallest integer $t \in \mathbb{N}^*$ such that every sequence $S$ over $G$ of length $|S| \geq t$ contains a non-empty zero-sum subsequence. 
The number $\mathsf{D}(G)$ is called the \em Davenport constant \em of $G$, and we refer to \cite{GeroKoch05,GeroRuzsa09} for background and connections with algebraic number theory.

\medskip
In what follows, we will need a classical theorem, independently proved in the late sixties by Kruyswijk \cite{EmdeBoas69} and Olson \cite{Olso69a}, stating that
\begin{equation}
\label{Davenport}
\mathsf{D}(C_{p^{\alpha_1}} \oplus \cdots \oplus C_{p^{\alpha_r}})=\displaystyle\sum^r_{i=1} \left(p^{\alpha_i}-1\right)+1
\end{equation}
for all primes $p$ and positive integers $\alpha_1, \dots, \alpha_r$. 
Our result now is the following.

\begin{proposition}
For every integer $r \geq 4$, there exists a finite Abelian group $G$ of rank $r$ for which 
$\mathsf{t}(G) > \eta(G)$.
\end{proposition}

\begin{proof}
Let $r \geq 4$ be an integer. 
It is an easy exercise to prove there is an integer $\alpha \geq 2$ such that
$$\frac{\ln(2r-1)}{\ln3} \leq \alpha \leq \frac{\ln(2^r-r)}{\ln3}.$$
Now, let us consider
$$G=C^{r-1}_3 \oplus C_{3^\alpha} \quad \text{ and } \quad H=C^r_3.$$ 
Since $G$ is a finite Abelian $3$-group, it follows from (\ref{Davenport}) that
$$\mathsf{D}(G)=2(r-1)+3^\alpha \leq 2\exp(G)-1,$$
so that \cite[Theorem $1.2$]{WolfiZhuang} yields
$$\eta(G) \leq \mathsf{D}(G)+\exp(G)-1 = 2r + 2.3^\alpha - 3.$$
On the other hand, we deduce from (\ref{lower bound}) that
$$\eta(H) \geq (2^r-1)(3-1)+1 = 2^{r+1}-1.$$
Therefore, $H$ is a subgroup of $G$ such that $\eta(H) > \eta(G)$, and the desired result follows from Lemma \ref{eta of a subgroup}.
\end{proof}

It would certainly be interesting to know whether the equality $\mathsf{t}(G)=\eta(G)$ holds for all finite Abelian groups of rank three.
In another direction, we would also like to address the following conjecture. 

\begin{conjecture}
\label{conjecture in the general case}
For all integers $r,n \in \mathbb{N}^*$, one has $\mathsf{t}(C^r_n)=\eta(C^r_n)$.
\end{conjecture}

It can readily be seen that Conjecture \ref{conjecture in the general case} holds whenever $G$ is an elementary $p$-group, since all non-zero elements of $G$ have same order in this case. 
In addition, our Corollary \ref{t(G) in the p-group case} already gives a "nearly optimal" answer when $G$ is of the form $C^r_n$, with $n$ a prime power.

\section*{Acknowledgments}
This work was started at IPAM in Los Angeles. 
Thus, I would like to warmly thank its staff, as well as the organizers of the program \textit{Combinatorics: Methods and Applications in Mathematics and Computer Science}, held in Fall 2009, for all their hospitality and for providing an excellent atmosphere for research. 
I would also like to thank W.~Schmid and the referees for helpful comments on an earlier version of this paper.


\begin{thebibliography}{50}

\bibitem{Alodub95}
{\sc N. Alon and M. Dubiner} 
{\it A lattice point problem and additive number theory}, 
Combinatorica {\bf 15} (1995), 301-309.

\bibitem{Chung89}
{\sc F. Chung} 
{\it Pebbling in hypercubes}, 
SIAM J. Discrete Math. {\bf 2} (1989), 467-472.

\bibitem{Denley}
{\sc T. Denley} 
{\it On a result of Lemke and Kleitman}, 
Combin. Probab. Comput. {\bf 6} (1997), 39-43.

\bibitem{EdelGero06}
{\sc Y. Edel, C. Elsholtz, A. Geroldinger, S. Kubertin and L. Rackham} 
{\it Zero-sum problems in finite abelian groups and affine caps}, 
Q. J. Math. {\bf 58} (2007), 159-186.

\bibitem{Elledge05}
{\sc S. Elledge and G. H. Hurlbert} 
{\it An application of graph pebbling to zero-sum sequences in abelian groups}, 
Integers {\bf 5} (2005), $\#$A17.

\bibitem{GaoGero06}
{\sc W. Gao and A. Geroldinger} 
{\it Zero-sum problems in finite abelian groups: a survey}, 
Expo. Math. {\bf 24} (2006), 337-369.

\bibitem{GaoSchmid06} 
{\sc W. Gao, Q. H. Hou, W. A. Schmid and R. Thangadurai} 
{\it On short zero-sum subsequences II}, 
Integers {\bf 7} (2007), $\#$A21.

\bibitem{Gero93}
{\sc A. Geroldinger} 
{\it On a conjecture of Kleitman and Lemke}, 
J. Number Theory {\bf 44} (1993), 60-65.

\bibitem{GeroRuzsa09}
{\sc A. Geroldinger} 
{\it Additive group theory and non-unique factorizations}, 
In A. Geroldinger and I. Ruzsa, {\it Combinatorial Number Theory and Additive Group Theory}, Advanced Courses in Mathematics, CRM Barcelona, Birkh\"{a}user (2009), 1-86.

\bibitem{GeroKoch05}
{\sc A. Geroldinger and F. Halter-Koch} 
{\it Non-unique factorizations. Algebraic, combinatorial and analytic theory}, 
Pure and Applied Mathematics {\bf 278}, Chapman \& Hall/CRC (2006).

\bibitem{GeroSchnCross94}
{\sc A. Geroldinger and R. Schneider} 
{\it The cross number of finite abelian groups II}, 
European J. Combin. {\bf 15} (1994), 399-405.

\bibitem{Girard09}
{\sc B. Girard} 
{\it A new upper bound for the cross number of finite Abelian groups}, 
Israel J. Math. {\bf 172} (2009), 253-278.

\bibitem{Hurlbert99}
{\sc G. Hurlbert} 
{\it A survey of graph pebbling}, 
Proceedings of the Thirtieth Southeastern International Conference on Combinatorics, Graph Theory and Computing {\bf 139} (1999), 41-64.

\bibitem{Hurlbert05}
{\sc G. Hurlbert} 
{\it Recent progress in graph pebbling}, 
Graph Theory Notes of New York {\bf 49} (2005), 25-37.

\bibitem{KleitLem89}
{\sc P. Lemke and D. Kleitman} 
{\it An addition theorem on the integers modulo $n$}, 
J. Number Theory {\bf 31} (1989), 335-345.

\bibitem{Olso69a}
{\sc J. E. Olson} 
{\it A combinatorial problem on finite abelian groups I}, 
J. Number Theory {\bf 1} (1969), 8-10.

\bibitem{WolfiZhuang}
{\sc W. A. Schmid and J. J. Zhuang} 
{\it On short zero-sum subsequences over $p$-groups}, 
Ars Combin. {\bf 95} (2010), 343-352.

\bibitem{EmdeBoas69}
{\sc P. van Emde Boas} 
{\it A combinatorial problem on finite abelian groups II}, 
Reports ZW-1969-007, Math. Centre, Amsterdam (1969).

\end{thebibliography}
\end{document}